\documentclass[amstex,12pt, amssymb]{article}

\usepackage{mathtext}
\usepackage[cp1251]{inputenc}
\usepackage[T2A]{fontenc}
\usepackage[dvips]{graphicx}
\usepackage{amsmath}
\usepackage{amssymb}
\usepackage{amsxtra}
\usepackage{latexsym}
\usepackage{ifthen}

\newcounter{lemma}[section]

\newcounter{corollary}[section]

\newcounter{rem}[section]

\newcounter{theorem}[section]

\newcounter{proposition}[section]

\numberwithin{equation}{section}

\begin{document}

\markboth{\centerline{R. SALIMOV AND E. SEVOST'YANOV}}
{\centerline{THE POLETSKII...}}

\def\cc{\setcounter{equation}{0}
\setcounter{figure}{0}\setcounter{table}{0}}

\overfullrule=0pt

%\normalsize\large

\author{{R. SALIMOV AND E. SEVOST'YANOV}\\}

\title{
{\bf THE POLETSKII AND V\"{A}IS\"{A}L\"{A} INEQUALITIES FOR THE
MAPPINGS WITH $(P, Q)$--DISTORTION}}

\date{\today}
\maketitle

%\large
\begin{abstract} The present paper is devoted to the study of
space mappings which are more general than quasiregular. Some
modulus inequalities for this class of mappings are obtained. In
particular, analogs of the well--known Poletskii and
V\"{a}is\"{a}l\"{a} inequalities were proved.
\end{abstract}


\section{Introduction}

The present paper is devoted to the study of quasiconformal mappings
and their generalizations, such as mappings with finite distortion
intensively investigated last time, see %\cite{AIM},
\cite{BGMV}, \cite{BGR}, \cite{Cr$_1$}--\cite{Cr$_2$},
\cite{Gol$_1$}--\cite{Gol$_2$}, \cite{IM}, \cite{IR},
 \cite{KO}, %\cite{MRV},
\cite{MRSY$_1$}--\cite{MRSY$_2$}, \cite{Mikl}, \cite{Pol},
\cite{Re$_1$}--\cite{Re$_2$}, \cite{Ri}, %\cite{Sal},
%\cite{Sev},
\cite{RSY$_1$}--\cite{RSY$_2$}, \cite{UV},
\cite{Va$_1$}--\cite{Va$_2$}.

Let us give some definitions. Everywhere below, $D$ is a domain in
${\Bbb R}^n,$ $n\ge 2,$ $m$ is the Lebesgue measure in ${\Bbb R}^n,$
$m(A)$ the Lebesgue measure of a measurable set $A\subset {\Bbb
R}^n,$ $m_1$ is the linear Lebesgue measure in ${\Bbb R}.$   A
mapping $f:D\rightarrow {\Bbb R}^n$ is {\it discrete} if $f^{-1}(y)$
consists of isolated points for each $y\in{\Bbb R}^n,$ and $f$ is
{\it open} if it maps open sets onto open sets. The notation
$f:D\rightarrow {\Bbb R}^n$ assumes that $f$ is continuous. In what
follows, a mapping $f$ is supposed to be orientation preserving,
i.e., the topological index $µ(y, f,G)>0$ for an arbitrary domain
$G\subset D$ such that $\overline{G}\subset D$ and $y\in
f(G)\setminus f(\partial G),$ see e.g. II.2 in \cite{Re$_2$}. Let
$f:D\rightarrow {\Bbb R}^n$ be a mapping and suppose that there is a
domain $G\subset D,$ $\overline{G}\subset D,$ for which $
f^{\,-1}\left(f(x)\right)=\left\{x\right\}.$ Then the quantity
$\mu(f(x), f, G),$ which is referred to as the local topological
index, does not depend on the choice of the domain $G$ and is
denoted by $i(x, f).$ Given a mapping $f:D\rightarrow {\Bbb R}^n,$ a
set $E\subset D$ and a point $y\in {\Bbb R}^n,$ we define the
multiplicity function $N(y, f, E)$ as the number of pre-images of
$y$ in $E,$ i.e., $$N(y, f, E) = {\rm card}\,\left\{x \in E: f(x) =
y\right\}\,$$
and
$$N(f,E)=\sup\limits_{y\in{\Bbb
R}^n}\,N(y,f,E)\,.$$

Recall that a mapping $f:D\rightarrow {\Bbb R}^n$ is said to have
{\it $N$--pro\-per\-ty (by Luzin)} if
$m\left(f\left(S\right)\right)=0$ whenever $m(S)=0$ for
$S\subset{\Bbb R}^n.$ Similarly, $f$ has the {\it
$N^{-1}$--pro\-per\-ty} if $m\left(f^{\,-1}(S)\right)=0$ whenever
$m(S)=0.$

A curve $\gamma$ in ${\Bbb R}^n$ is a continuous mapping $\gamma
:\Delta\rightarrow{\Bbb R}^n$ where $\Delta$ is an interval in
${\Bbb R} .$ Its locus $\gamma(\Delta)$ is denoted by $|\gamma|.$
Given a family $\Gamma$ of curves $\gamma$ in ${\Bbb R}^n ,$ a Borel
function $\rho:{\Bbb R}^n \rightarrow [0,\infty]$ is called {\it
admissible} for $\Gamma ,$ abbr. $\rho \in {\rm adm}\, \Gamma ,$ if
$$\int\limits_{\gamma} \rho(x)|dx| \ge 1$$ for each (locally
rectifiable) $\gamma\in\Gamma.$ Given $p\ge 1,$ the {\it
$p$--mo\-du\-lus} of $\Gamma$ is defined as the quantity
$$M_p(\Gamma):=\inf\limits_{ \rho \in {\rm adm}\, \Gamma}
\int\limits_{{\Bbb R}^n} \rho^p(x) dm(x)$$ interpreted as $+\infty$
if ${\rm adm}\, \Gamma = \varnothing .$ Note that
$M_p(\varnothing)=0;$ $M_p(\Gamma_1)\le M_p(\Gamma_2)$ whenever
$\Gamma_1\subset\Gamma_2,$ and
$M_p\left(\bigcup\limits_{i=1}^{\infty}\Gamma_i\right)\le
\sum\limits_{i=1}^{\infty}M_p(\Gamma_i),$ see Theorem 6.2 in
\cite{Va$_1$}.

We say that a property $P$ holds for {\it $p$--almost every
($p$--a.e.)} curves $\gamma$ in a family $\Gamma$ if the subfamily
of all curves in $\Gamma $, for which $P$ fails, has $p$--mo\-du\-lus
zero.

If $\gamma :\Delta\rightarrow{\Bbb R}^n$ is a locally rectifiable
curve, then there is the unique nondecreasing length function
$l_{\gamma}$ of $\Delta$ onto a length interval $\Delta
_{\gamma}\subset{\Bbb R}$ with a prescribed normalization $l
_{\gamma}(t_0)=0\in\Delta _{\gamma},$ $t_0\in\Delta,$ such that $l
_{\gamma}(t)$ is equal to the length of the subcurve $\gamma
|_{[t_0,t]}$ of $\gamma$ if $t>t_0,$ $t\in\Delta ,$ and $l
_{\gamma}(t)$ is equal to minus length of $\gamma |_{[t,t_0]}$ if
$t<t_0,$ $t\in\Delta .$ Let $g: |\gamma |\rightarrow{\Bbb R}^n$ be a
continuous mapping, and suppose that the curve $\widetilde{\gamma}
=g\circ\gamma$ is also locally rectifiable. Then there is a unique
non--decreasing function $L_{\gamma ,g}: \Delta
_{\gamma}\rightarrow\Delta _{\widetilde{\gamma}}$ such that
$L_{\gamma ,g}\left(l_{\gamma}(t)\right) =
l_{\widetilde{\gamma}}(t)$ for all $t\in\Delta.$ A curve $\gamma$ in
$D$ is called here a (whole) {\it lifting} of a curve
$\widetilde{\gamma}$ in ${\Bbb R}^n$ under $f:D\rightarrow {\Bbb
R}^n$ if $\widetilde{\gamma} = f\circ\gamma.$

\medskip
We say that a mapping $f:D\rightarrow{\Bbb R}^n$ satisfies the {\it
$L$--pro\-pe\-r\-ty with respect to $(p, q)$-- mo\-du\-lus}, iff the
following two conditions hold:

$\left(L^{(1)}_p\right):$ for $p$--a.e. curve $\gamma$ in $D,$
$\widetilde{\gamma}=f\circ\gamma$ is locally rectifiable and the
function $L_{\gamma ,f}$ has the $N$--pro\-per\-ty;

$\left(L^{(2)}_q\right):$ for $q$--a.e. curve $\widetilde{\gamma}$
in $f(D),$ each lifting $\gamma$ of $\widetilde{\gamma}$ is locally
rectifiable and the function $L_{\gamma ,f}$ has the
$N^{-1}$--pro\-per\-ty.

The notation $(p,q)$-modulus mentioned above means that the $L^{(1)}_p$ and $L^{(2)}_q$ properties hold with respect to $p$ and $q$-moduli, respectively.

\medskip
A mapping $f:D\rightarrow{\Bbb R}^n$ is called a mapping with {\it
finite length $(p, q)$--dis\-tor\-tion}, if $f$ is differentiable
a.e. in $D,$ has $N$-- and $N^{\,-1}$--pro\-per\-ties, and
$L$--pro\-pe\-r\-ty with respect to $(p, q)$-- mo\-du\-lus. The
mappings of finite length $(p, q)$--dis\-tor\-tion are natural
generalization of the mappings with finite length distortion
introduced in the paper \cite{MRSY$_1$}, see also monograph
\cite{MRSY$_2$}.

Set at points $x\in D$ of differentiability of $f$
$$l\left(f^{\,\prime}(x)\right)\,=\,\min\limits_{h\in {\Bbb
R}^n \backslash \{0\}} \frac {|f^{\,\prime}(x)h|}{|h|}\,,
\Vert f^{\,\prime}(x)\Vert\,=\,\max\limits_{h\in {\Bbb R}^n
\backslash \{0\}} \frac {|f^{\,\prime}(x)h|}{|h|}\,, J(x,f)=det f'(x),$$
and define for any $x\in D$ and fixed $p,q$, $p,q\geq 1$
$$K_{I, q}(x,f)\quad =\quad\left\{
\begin{array}{rr}
\frac{|J(x,f)|}{{l\left(f^{\,\prime}(x)\right)}^q}, & J(x,f)\ne 0,\\
1,  &  f^{\,\prime}(x)=0, \\
\infty, & {\rm otherwise}
\end{array}
\right.\,,$$
$$K_{O, p}(x,f)\quad =\quad \left\{
\begin{array}{rr}
\frac{\Vert f^\prime(x)\Vert^p}{|J(x,f)|}, & J(x,f)\ne 0,\\
1,  &  f^{\,\prime}(x)=0, \\
\infty, & {\rm otherwise}
\end{array}
\right.\,.$$

\medskip
One of the main results proved in the paper is following.

\medskip
\begin{theorem}\label{th1A} {\sl\, A mapping $f:D\rightarrow {\Bbb R}^n$ with finite length
$(p, q)$--dis\-tor\-ti\-on satisfies the inequalities
\begin{equation}\label{eq5*AA}
M_q(f(\Gamma))\le \int\limits_D K_{I, q}(x,f)\cdot\rho^q(x)\, dm(x)
\end{equation}
for every family of curves $\Gamma$ in $D$ and $\rho\in {\rm
adm\,}\Gamma,$ and
\begin{equation}\label{eq5.4AA}
M_p(\Gamma)\le \int\limits_{f(E)}K_{I,p}\left(y, f^{\,-1},
E\right)\cdot\rho_*^p(y)dm(y) \end{equation} for every measurable
set $E\subset D,$ every family $\Gamma$ of curves  $\gamma$ in $E$
and every function $\rho_*(y)\in {\rm adm\,}f(\Gamma),$ where
\begin{equation}\label{eq1}
K_{I, p}\left(y, f^{\,-1}, E\right):=\sum\limits_{x\in E\cap
f^{\,-1}(y)}K_{O, p} (x,f)\,.\end{equation} }
\end{theorem}

Remark that an analog of the Theorem \ref{th1A} for $p=q=n$ was
proved in \cite{MRSY$_2$}, see Theorems 8.5 and 8.6 (cf. \cite{MRSY$_1$}).

\section{The proof of the main results}
\setcounter{equation}{0}

Further, we use the notation $I$ for the segment $[a,b].$ Given a
closed rectifiable path $\gamma:I\rightarrow {\Bbb R}^n,$ we define
a length function $l_{\gamma}(t)$ by the rule
$l_{\gamma}(t)=S\left(\gamma, [a,t]\right),$ where $S(\gamma,
[a,t])$ is the length of the path $\gamma|_{[a, t]}.$ Let
$\alpha:[a,b]\rightarrow {\Bbb R}^n$ be a rectifiable curve in
${\Bbb R}^n,$ $n\ge 2,$ and $l(\alpha)$ be its length. A {\it normal
representation} $\alpha^0$ of $\alpha$ is defined as a curve
$\alpha^0:[0, l(\alpha)]\rightarrow {\Bbb R}^n$ which can be got
from $\alpha$ by change of parameter such that
$\alpha(t)=\alpha^0\left(S\left(\alpha, [a, t]\right)\right)$ for
every $t\in [0, l(\alpha)].$

\medskip
Suppose that $\alpha$ and $\beta$ are curves in ${\Bbb R}^n.$ Then the
notation $\alpha\subset\beta$ denotes that $\alpha$ is a subpath of
$\beta.$ In what follows, $I$ denotes either an open or  closed or
semi--open interval on the real axes. The following definition can
be found in the section 5 of Ch. II in \cite{Ri}.

\medskip
Let $f:D\rightarrow {\Bbb R}^n$ be a mapping such that $f^{-1}(y)$
does not contain a non--degenerate curve, $\beta:I_0\rightarrow
{\Bbb R}^n$ be a closed rectifiable curve and $\alpha:I\rightarrow
D$ be such that $f\circ \alpha\subset \beta.$ If the length function
$l_{\beta}:I_0\rightarrow [0, l(\beta)]$ is a constant on $J\subset
I,$ then $\beta$ is a constant on $J$ and, consequently, the curve
$\alpha$ to be a constant on $J.$ Thus, there is a unique function
$\alpha^{\,*}:l_\beta(I)\rightarrow D$ such that
$\alpha=\alpha^{\,*}\circ (l_\beta|_I).$ We say that $\alpha^{\,*}$
to be a {\it $f$--re\-pre\-se\-n\-ta\-tion of $\alpha$ with respect
to $\beta$ } if $\beta=f\circ\alpha.$

\medskip
\begin{rem}\label{rem1}
Given a closed rectifiable curve $\gamma:[a, b]\rightarrow {\Bbb
R}^n$ and $t_0\in (a, b),$ let $l_{\gamma}(t)$ denote the length of
the subcurve $\gamma |_{[t_0,t]}$ of $\gamma$ if $t>t_0,$ $t\in (a,
b),$ and $l_{\gamma}(t)$ is equal to the length of $\gamma |_{[t,t_0]}$
with the sign $"$$-$$"$ if $t<t_0,$ $t\in (a, b).$ Then we observe
that properties of the $L_{\gamma, f}$ connected with the length
functions $l_{\gamma}(t)$ and $l_{\widetilde{\gamma}}(t),$
$\widetilde{\gamma}=f\circ\gamma,$ do not essentially depend on the
choice of $t_0\in (a, b).$ Moreover, we may consider that in this
case $t_0=a$ because given $t_0\in (a, b),$ $S(\gamma, [a,
t])=S(\gamma, [a, t_0])+l_{\gamma}(t).$ Hence further we choose
$t_0=a$ and use the notion $l_{\gamma}(t)$ for the length of the
path $\gamma|_{[a, t]}$ whenever a curve $\gamma$ is closed.
\end{rem}

\medskip
The following statement gives the connection between $L^{(1)}_p$--
and $L^{(2)}_q$--pro\-per\-ti\-es and some properties of curves
meaning above.

\medskip
\begin{proposition}\label{pr3}
{\sl\, $(i_1)\quad$ A mapping $f:D\rightarrow {\Bbb R}^n$ has
$L^{(1)}_p$--pro\-per\-ty if and only if the curve $f\circ\gamma^0$
is rectifiable and absolutely continuous for $p$--a.e. closed curve
$\gamma;$ $(i_2)\quad$ a mapping $f:D\rightarrow {\Bbb R}^n$ has
$L^{(2)}_q$--pro\-per\-ty if and only if $f^{-1}(y)$ does not
contain a nondegenerate curve for every $y\in {\Bbb R}^n,$ and the
$f$--representation $\gamma^{\,*}$ is rectifiable and absolutely
continuous for $q$--a.e. closed curve
$\widetilde{\gamma}=f\circ\gamma.$}
\end{proposition}

\begin{proof}$(i_1)$ First assume that $f$ has $L^{(1)}_p$--pro\-per\-ty.
Then for $p$--a.e. curve $\gamma$ the curve $f\circ \gamma$ is
locally rectifiable and, thus, $f\circ \gamma^0$ is rectifiable for
$p$--a.e. closed curve $\gamma$ because
$\left(f\circ\gamma^{0}\right)^0=\left(f\circ\gamma\right)^0,$ see
Theorem 2.6 in \cite{Va$_1$}. By $L^{(1)}_p$--pro\-per\-ty
$L_{\gamma, f}$ has $N$--pro\-per\-ty for $p$--a.e. curve $\gamma$
in $D$ that is equivalent to absolute continuity of $L_{\gamma, f},$
see Theorem 2.10.13 in \cite{Fe}. By definition of the length
function, we have that $\alpha(t)=\alpha^{\,0}\circ l_{\alpha}(t)$
for each locally rectifiable curve $\alpha$ in ${\Bbb R}^n.$ Thus,
in particular, for $\alpha=f\circ\gamma^{\,0},$ we obtain
$$f\circ\gamma^{\,0}(s)=\left(f\circ\gamma^{\,0}\right)^{\,0}\circ
l_{f\circ\gamma^{\,0}}(s)=\left(f\circ\gamma\right)^{\,0}\circ
l_{f\circ\gamma^{\,0}}(s)=\left(f\circ\gamma\right)^{\,0}\circ
L_{\gamma, f}(s)\,\quad \forall\quad s\in [0, l(\gamma)]\,.$$
Since $L_{\gamma, f}$ is absolutely continuous and
$|\left(f\circ\gamma\right)^0(s_1)-\left(f\circ\gamma\right)^0(s_2)|\le
|s_1-s_2|$ for every $s_1, s_2\in [0, l(f\circ \gamma)],$ the curve
$f\circ\gamma^{\,0}$ is absolutely continuous for $p$--a.e. curve
$\gamma.$

Inversely, let us assume that the curve $f\circ\gamma^0$ is
rectifiable and absolutely continuous for $p$--a.e. closed curve
$\gamma.$ Note that $L_{\gamma, f}=l_{f\circ \gamma^{\,0}}$ for such
$\gamma.$ Let $\Gamma_1$ be a family of all closed curves $\alpha$
in $D$ such that $f\circ\alpha$ either is not rectifiable or
$L_{\alpha, f}$ is not absolutely continuous. Let $\Gamma$ be a
family of all curves $\gamma$ in $D$ such that $f\circ\gamma$ either
is not locally rectifiable or $L_{\gamma, f}$ is not locally
absolutely continuous. Then $\Gamma>\Gamma_1$ and, thus,
$M_p(\Gamma)\le M_p(\Gamma_1),$ i.e., $M_p(\Gamma)=0.$

$(i_2)$ Now, let us assume that $f$ has $L^{(2)}_q$--pro\-per\-ty.
Then $\gamma^{\,*}$ is rectifiable for $q$--a.e. closed curve
$\widetilde{\gamma}$ whenever $\widetilde{\gamma}=f\circ \gamma$
because $(\gamma^{\,*})^{\,0}=\gamma^{\,0},$ see Theorem 2.6 in
\cite{Va$_1$}. Moreover, we observe that $f^{-1}(y)$ does not
contain a nondegenerate curve for every $y\in {\Bbb R}^n$ because
$L_{\gamma, f}$ has $N^{\,-1}$--pro\-per\-ty for $q$--a.e. closed
curve $\widetilde{\gamma}$ and all $\gamma$ with
$\widetilde{\gamma}=f\circ \gamma.$ Thus, $L_{\gamma, f}^{\,-1}$ is
well--defined and, for such $\gamma$ and $\widetilde{\gamma},$ we
have
$$\gamma^{\,*}\circ l_{\widetilde{\gamma}}(t)=\gamma(t)=\gamma^{\,0}\circ
l_{\gamma}(t)=\gamma^{\,0}\circ L_{\gamma,
f}^{\,-1}\left(l_{\widetilde{\gamma}}(t)\right)$$
and, denoting $s:=l_{\widetilde{\gamma}}(t),$ we obtain
$$\gamma^{\,*}(s)=\gamma^{\,0}\circ L_{\gamma,
f}^{\,-1}(s)\,.$$
Thus $\gamma^{\,*}$ is absolutely continuous because $L_{\gamma,
f}^{\,-1}(s)$ is absolutely continuous, see Theorem 2.10.13 in
\cite{Fe}, and
$$|\gamma^{\,0}(s_1)-\gamma^{\,0}(s_2)|\le |s_1-s_2|$$
for all $s_1, s_2\in [0, l(\gamma)].$

Inversely, let us assume that $f^{-1}(y)$ does not contain a
nondegenerate curve for every $y\in {\Bbb R}^n$ and the curve
$\gamma^{\,*}$ is rectifiable and absolutely continuous for
$q$--a.e. closed curve $\widetilde{\gamma}=f\circ \gamma.$ Then
$L^{\,-1}_{\gamma, f}$ is well--defined for $q$--a.e. closed curve
$\widetilde{\gamma}$ and all $\gamma$ with
$\widetilde{\gamma}=f\circ \gamma.$ By Theorem 2.6 in \cite{Va$_1$}
$\gamma^{\,*\,0}=\gamma^{\,0}.$ Moreover, for all such
$\widetilde{\gamma},$ $\gamma$ and $\gamma^{\,*},$
$l_{\gamma^{\,*}}(s)=L_{\gamma, f}^{\,-1}(s),$ and absolutely
continuity of $L_{\gamma, f}^{\,-1}(s)$ follows from Theorem 1.3 in
\cite{Va$_1$}. Let $\Gamma_1$ be the family of all closed curves
$\widetilde{\alpha}=f\circ\alpha$ in $f(D)$ such that $\alpha^{\,*}$
either is not rectifiable or $L_{\alpha, f}^{\,-1}(s)$ is not
absolutely continuous. By the assumption $M_q(\Gamma_1)=0.$ Let
$\Gamma$ be a family of all curves $\widetilde{\gamma}=f\circ\gamma$
in $f(D)$ such that $\gamma$ either is not locally rectifiable or
$L_{\gamma, f}^{\,-1}(s)$ is not locally absolutely continuous. Then
$\Gamma>\Gamma_1$ and, thus, $M_q(\Gamma)\le M_q(\Gamma_1)$ that
implies the desired equality $M_q(\Gamma)=0.$
\end{proof}$\Box$

\medskip
A mapping $\varphi:X\rightarrow Y$ between metric spaces $X$
and $Y$ is said to be a {\it Lipschitzian} provided
$$
{\rm dist} \left(\varphi(x_1),\varphi(x_2)\right)\le
M\cdot\text{dist} (x_1,x_2)
$$
for some $M<\infty$ and for all $x_1$ and $x_2\in X.$ The mapping
$\varphi$ is called {\it bi--lipschitz} if, in addition,
$$
M^*\text{dist} \left(x_1,x_2\right)\le\text{dist}
\left(\varphi\left(x_1\right),\varphi\left(x_2\right)\right)
$$ for some $M^*>0$ and for all $x_1$ and $x_2\in X.$ In what follows, $X$
and $Y$ are subsets of ${\Bbb R}^n$ with the Euclidean distance.

The following proposition can be found in \cite{MRSY$_1$}, see Lemma
3.20, see also Lemma 8.3 in \cite{MRSY$_2$}.

\medskip
\begin{proposition}\label{pr1}{\sl\,
Let $f:D\rightarrow{\Bbb R}^n$ be differentiable a.e. in $D$ and
have $N$-- and $N^{-1}$--pro\-per\-ties. Then there is a countable
collection of compact sets $C^*_k\subset D$ such that $m(B_0)=0$
where $B_0=D\setminus\bigcup\limits_{k=1}\limits^{\infty} C^*_k$ and
$f|_{C^*_k}$ is one--to--one and bi--lipschitz for every
$k=1,2,\ldots ,$ and, moreover, $f$ is differentiable at points of
$C_k^*$ with $J(x,f)\ne 0.$}
\end{proposition}

\medskip
Given a set $E$ in ${\Bbb R}^n$ and a curve $\gamma
:\Delta\rightarrow {\Bbb R}^n,$ we identify $\gamma\cap E$ with
$\gamma\left(\Delta\right)\cap E.$
%where $|\gamma |=\gamma (\Delta)$ is the locus of $\gamma .$
If $\gamma$ is locally rectifiable, then we set
$$
l\left(\gamma\cap E\right) =  m_1(E_ {\gamma}),
$$ where
$E_ {\gamma} = l_{\gamma}\left(\gamma ^{-1}\left(E\right)\right);$
here $l_{\gamma}:\Delta\rightarrow\Delta
_{\gamma}$ as in the previous section. Note that %
$E_ {\gamma} = \gamma _0^{-1}\left(E\right),$
where $\gamma _0 :\Delta _{\gamma}\rightarrow {\Bbb R}^n$ is the
natural parametrization of $\gamma $ and
$$l\left(\gamma\cap E\right) = \int\limits_{\Delta} \chi
_E\left(\gamma\left(t\right) \right) |dx|:= \int\limits_{\Delta
_{\gamma}} \chi_{E_\gamma }(s)\, ds\,.$$

\medskip
The following statement can be found in \cite{MRSY$_2$}, see Theorem
9.1 for $k=1.$

\medskip
\begin{proposition}\label{pr2}{\sl\,
Let $E$ be a set in a domain $D\subset{\Bbb R}^n,$ $n\ge 2,$ $p\ge
1.$ Then $E$ is measurable if and only if $\gamma\cap E$ is
measurable for $p$--a.e. curve $\gamma$ in $D.$ Moreover, $m(E)=0$
if and only if
%\begin{equation}\label{eq2}
$$l(\gamma\cap E)=0$$ %\end{equation}
on $p$--a.e. curve $\gamma$ in $D.$}
\end{proposition}

\medskip
The following result was proved in \cite{MRSY$_1$}--\cite{MRSY$_2$}
for a case $p=n.$ Here we extend the study of this problem to arbitrary $p$, $p\ge 1$.

\medskip
\begin{theorem}\label{th1}
{\sl\, Let $f:D\rightarrow {\Bbb R}^n$ be differentiable a.e. in
$D,$ have  $N$ -- and $N^{-1}$--pro\-per\-ties and
$L^{(1)}_p$--pro\-per\-ty. Then  relation (\ref{eq5.4AA})  holds for every measurable set $E\subset D,$ every family $\Gamma\subset
E$ of curves  $\gamma$ in $E$ and any  $\rho_*(y)\in {\rm
adm\,}f(\Gamma),$ where $K_{I, p}\left(y, f^{\,-1}, E\right)$ is
defined by (\ref{eq1}).}
\end{theorem}

\medskip
\begin{proof} By Theorem III.6.6 (iV)  \cite{Sa},
$E=B\cup B_0,$ where $B$ is a set of the class $F_{\sigma},$ and
$m(B_0)=0.$ Consequently, $f(E)$ is measurable by $N$--pro\-per\-ty
of $f.$ Without loss of generality, we may assume that $f(E)$ is a
Borel set and that $\rho_*\equiv 0$ outside of $f(E).$ In other case
we can find a Borel set $G$ such that $f(E)\subset G$ and
$m(G\setminus f(E))=0,$ see (ii) of the Theorem III.6.6 in
\cite{Sa}. Now, a set $f^{-1}(G)$ is  Borel and $E\subset
f^{-1}(G).$ Note that in this case the function
$$
\rho^G_* (y)= \left \{\begin{array}{rr}
\rho_*(y), &  {\rm for }\,\,\, y\in G, \\
0, & y\in \overline{{\Bbb R}^n}\setminus G\end{array} \right.$$
is a Borel function, as well.  Now suppose that $f(E)$ is a Borel
set. Let $B_0$ and $C_k^*,$ $k=1,2,\ldots ,$ be as in Proposition
\ref{pr1}. Setting by induction $B_1=C_1^*,$ $B_2=C_2^*\setminus
B_1,\ldots ,$ and
\begin{equation} \label{eq7.3.7y} B_k=C_k^*\setminus
\bigcup\limits_{l=1}\limits^{k-1}B_l
\end{equation} we obtain a countable covering of $D$ consisting of
mutually disjoint Borel sets $B_k, k=0,1,2,\ldots $ with $m(B_0)=0,$
$B_0=D\setminus \bigcup\limits_{k=1}^{\infty} B_k.$
Remark that $\gamma^{\,0}(s)\not\in B_0$ for a.e. $s$ and $p$--a.e.
closed curve $\gamma\in \Gamma,$ see Proposition \ref{pr2}; here
$\gamma^{0}(s)$ denotes a normal representation of $\gamma.$ By
Proposition \ref{pr3} a curve $f\circ\gamma^{0}$ is rectifiable and
absolutely continuous for $p$--a.e. $\gamma\in\Gamma.$
\medskip

Let $\Vert f^{\,\prime}(x)\Vert=\max \{ |f^{\,\prime}(x)h|: h \in
{\Bbb R}^n, |h|=1\}.$ Given $\rho_*\in{\rm adm}\,f(\Gamma),$ set
%
%\begin{equation} \label{eq7.3.12g} \
$$\rho (x)= \left \{\begin{array}{rr}
\rho_*(f(x))\Vert f^{\,\prime}(x)\Vert, &  {\rm for }\quad x\in D\setminus B_0, \\
0, & {\rm otherwise.}\end{array} \right. %\end{equation}
$$
By Theorem 5.3 in \cite{Va$_1$} (see also Lemma II.2.2 in \cite{Ri})
we obtain that
$$\int\limits_{\gamma} \rho(x)
|dx|=\int\limits_{\gamma} \rho_*(f(x))\Vert f^{\,\prime}(x)\Vert
|dx|\ge $$
\begin{equation} \label{eq7.3.7b}
\ge \int\limits_{f\circ\gamma} \rho_*(y) |dy| \ge 1
\end{equation}
for $p$--a.e. closed $\gamma\in\Gamma,$ i.e., $\rho\in{\rm
adm}\,\Gamma.$ The case of arbitrary $\alpha\in \Gamma$ can be
gotten from (\ref{eq7.3.7b}) by taking of $\sup$ over all closed
subpaths $\gamma\subset \alpha.$
Therefore,
\begin{equation} \label{eq7.3.7z} M_p(\Gamma)\le\int\limits_{D}
\rho^p(x) dm(x). \end{equation}
Note that $\rho =\sum\limits_{k=1}\limits^{\infty}\rho_k$, where
$\rho_k = \rho\cdot\chi_{B_k}$ have mutually disjoint supports. By
3.2.5 for $m=n$ in \cite{Fe}
we obtain that
$$\int\limits_{f(B_k\cap E)} K_{O, p}\left(f_k^{-1}(y),
f)\cdot\rho_*^p(y\right) dm(y)= \int\limits_{B_k} K_{O, p}(x,
f)\rho_*^p\left(f(x)\right)|J(x, f)|dm(x)=$$
\begin{equation}
\label{eq7.3.7x} =\int\limits_{B_k} \Vert f^{\,\prime}(x)\Vert^p
\rho_*^p\left(f(x)\right)dm(x)=\int\limits_{D}\rho_k^p(x) dm(x)\,,
\end{equation} where every $f_k=f|_{B_k},$ $k=1,2,\ldots $ is
injective by the construction.

\medskip
Finally, by the Lebesgue positive convergence theorem, see e.g.
Theorem I.12.3 in \cite{Sa}, we conclude from (\ref{eq7.3.7z}) and
(\ref{eq7.3.7x}) that $$\int\limits_{f(E)} K_{I,p}(y, f^{-1},
E)\cdot\rho_*^p(y)dm(y)= \int\limits_{D}
\sum\limits_{k=1}\limits^{\infty}\rho_k^p(x)dm(x) \ge
M_p(\Gamma)\,.\,\Box$$
\end{proof}

\medskip
The following result is a generalization of the known Poletskii
inequality for quasiregular mappings, see Theorem 1 in \cite{Pol}
and Theorem II.8.1 in \cite{Ri}. Its analog was also proved in
\cite{MRSY$_1$}--\cite{MRSY$_2$} for the case $q=n.$

\medskip
\begin{theorem}\label{th2}
{\sl\, Let a mapping $f:D\rightarrow {\Bbb R}^n$ be differentiable
a.e. in $D,$ have $N$-- and $N^{-1}$--pro\-per\-ties, and
$L^{(2)}_q$--pro\-per\-ty, too. Then  relation (\ref{eq5*AA})
holds for every curve family $\Gamma$ in $D$ and any function
$\rho\in{\rm adm}\,\Gamma.$}
\end{theorem}

\medskip
\begin{proof} Let $B_k,$ $k=0,1,2,\ldots ,$ be given as above by
(\ref{eq7.3.7y}). By the assumption, $f$ has $N$--pro\-per\-ty in
$D$ and, consequently, $m(f(B_0))=0.$ Let $\rho\in{\rm adm}\,\Gamma
$ and
%
%
%\begin{equation}\label{eq7.3.13}
%
$$\widetilde{\rho}(y)=\chi_{f(D\setminus B_0)}\cdot
\sup\limits_{x\in f^{-1}(y)\cap D\setminus B_0}
\rho^* (x)\,,$$ %\end{equation}
where
%
%\begin{equation}\label{eq7.3.13v}
$$\rho^{\,*} (x)= \left \{\begin{array}{rr}\rho
(x)/l\left(f^{\,\prime}(x)\right),
&  {\rm for }\quad x\in D\setminus B_0, \\
0, & {\rm otherwise.}\end{array} \right.$$
%\end{equation}
%
%
Note that $\widetilde{\rho}(y)=\sup\limits_{k\in{\Bbb N}}\rho_k(y)$,
where
%\begin{equation}\label{eq7.3.13a}
$$\rho_k(y)= \left \{\begin{array}{rr}
\rho^*(f^{-1}_k(y)), &  {\rm for }\,\,\,  y\in f(B_k), \\
0, & {\rm otherwise,}\end{array} \right.
%\end{equation}
$$ and every $f_k=f|_{B_k},$ $k=1,2,\ldots,$ is injective. Thus, the
function $\widetilde{\rho}$ is Borel, see section 2.3.2 in
\cite{Fe}.
\medskip

Let $\widetilde{\gamma}$ be a closed rectifiable curve such that
$\widetilde{\gamma}=f\circ \gamma,$ ${\widetilde{\gamma}}^0$ be a
normal representation of $\widetilde{\gamma}$ and $\gamma^*$ be
$f$--re\-pre\-sen\-ta\-tion of $\gamma$ by the respect to
$\widetilde{\gamma},$ see above. Since $m(f(B_0))=0,$
$\widetilde{\gamma}^0(s)\not\in f(B_0)$ for $q$--a.e. curve
$\widetilde{\gamma}$ and for a.e. $s\in [0, l(\widetilde{\gamma})],$
see Proposition \ref{pr2}. For $q$--a.e. paths $\widetilde{\gamma}$
and all $\gamma$ with $\widetilde{\gamma}=f\circ \gamma$ we have
$$\int\limits_{\widetilde{\gamma}} \widetilde{\rho}(y)|dy|=
\int\limits_{0}^{l(\widetilde{\gamma})}
\widetilde{\rho}({\widetilde{\gamma}}^0(s))\,ds\,=
$$
\begin{equation}\label{eq2}
=\int\limits_{0}^{l(\widetilde{\gamma})} \sup\limits_{x\in
f^{\,-1}(\widetilde{\gamma}^0(s))\cap D\setminus B_0}
\rho^{\,*}(x)\,\,ds\ge \int\limits_{0}^{l(\widetilde{\gamma})}
\frac{\rho(\gamma^{\,*}(s))}{l\left(
f^{\,\prime}(\gamma^{\,*}(s))\right)}\,ds\,.
\end{equation}

Since ${\widetilde{\gamma}}^{\,0}$ is rectifiable,
${\widetilde{\gamma}}^{\,0}(s)$ is differentiable a.e. Besides that,
a curve $\gamma^{\,*}$ is absolutely continuous for $q$--a.e.
$\widetilde{\gamma}$ by Proposition \ref{pr3}. Since
$\widetilde{\gamma}^0(s)\not\in f(B_0)$ for a.e. $s\in [0,
l(\widetilde{\gamma})]$ and $q$--a.e. curve $\widetilde{\gamma},$ we
have $\gamma^{\,*}(s)\not\in B_0$ for a.e. $s\in [0,
l(\widetilde{\gamma})].$ Thus, the derivatives
$f^{\,\prime}\left(\gamma^{\,*}(s)\right)$ and
$\gamma^{\,*\prime}(s)$ exist for a.e. $s.$ Taking into account the
formula of the derivative of the superposition of functions, and
that the modulus of the derivative of the curve by the natural
parameter equals to 1, we have that
$$1=\left|\left(f\circ {\gamma}^{\,*}\right)^{\,\prime}(s)\right|=
\left|f^{\,\prime}\left(\gamma^{\,*}(s)\right)\gamma^{\,*\prime}(s)\right|=$$
\begin{equation}\label{eq2A}
=\left|f^{\,\prime}\left(\gamma^{\,*}(s)\right)
\cdot\frac{\gamma^{\,*\prime}(s)}{|\gamma^{\,*\prime}(s)|}\right|\cdot
|\gamma^{\,*\prime}(s)|\ge
l\left(f^{\,\prime}\left(\gamma^{\,*}(s)\right)\right)\cdot
|\gamma^{\,*\prime}(s)|\,.\end{equation}
It follows from (\ref{eq2A}) that a.e.
\begin{equation}\label{eq3A}
\frac{\rho(\gamma^*(s))}
{l\left(f^{\,\prime}\left(\gamma^{\,*}(s)\right)\right)}\ge
\rho(\gamma^*(s))\cdot |\gamma^{\,*\prime}(s)|\,.
\end{equation}
By absolute continuity of $\gamma^{\,*},$ definition of $\rho$ and
Theorem 4.1 in \cite{Va$_1$} we obtain
\begin{equation}\label{eq4A}
1\le
\int\limits_{\gamma}\rho(x)|dx|=\int\limits_{0}^{l(\widetilde{\gamma})}
\rho\left(\gamma^{\,*}(s)\right) \cdot
|\gamma^{\,*\prime}(s)|\,ds\,.
\end{equation}
It follows from (\ref{eq2}), (\ref{eq3A}) and (\ref{eq4A}) that
$\int\limits_{\widetilde{\gamma}}\widetilde{\rho}(y)|dy|\ge 1$ for
$q$--a.e. closed curve $\widetilde{\gamma}$ in $f(\Gamma).$ The case
of the arbitrary path $\widetilde{\gamma}$ can be got from the
taking of $\sup$ in
$\int\limits_{\widetilde{\gamma}^{\,\prime}}\widetilde{\rho}(y)|dy|\ge
1$ over all closed subpaths $\widetilde{\gamma}^{\,\prime}$ of
$\widetilde{\gamma}.$ Thus, $\widetilde{\rho}(y)\in{\rm
adm}\,f(\Gamma)\setminus \Gamma_0,$ where $M_q(\Gamma_0)=0.$ Hence
\begin{equation}\label{eq3}
M_q\left(f\left(\Gamma\right)\right)\le \int\limits_{f(D)}
{\widetilde{\rho}}^{\,q}(y) dm(y)\,.
\end{equation}
Further, by 3.2.5 for $m=n$ in \cite{Fe} we have that
$$\int\limits_{B_k} K_{I,q}(x,f)\cdot\rho^q(x)dm(x)=
\int\limits_{B_k}\frac{|J(x,
f)|}{\left(l\left(f^{\,\prime}(x)\right)\right)^q}\cdot
\rho^q(x)dm(x)=$$
\begin{equation} \label{eq7.3.14}
=\int\limits_{f(B_k)}\frac{\rho^q\left(f_k^{-1}(y)\right)}
{\left(l\left(f^{\,\prime}\left(f_k^{\,-1}(y)\right)\right)\right)^q}\,dm(y)=
\int\limits_{f(D)}\rho_k^q(y)dm(y)\,.
\end{equation}
Finally, by the Lebesgue theorem, see Theorem 12.3 $\S\,12$ of Ch. I
in \cite{Sa}, we obtain from (\ref{eq3}) and (\ref{eq7.3.14}) the
desired inequality
$$\int\limits_{D} K_{I,q}(x,f)\cdot\rho^q(x)dm(x)=
\sum\limits_{k=1}\limits^{\infty}\int\limits_{B_k}
K_{I,q}(x,f)\cdot\rho^q(x)dm(x)=$$
$$=\int\limits_{f(D)}\sum\limits_{k=1}\limits^{\infty}
\rho_k^q(y)dm(y)\ge \int\limits_{f(D)}\sup\limits_{k\in {\Bbb N}}
\rho_k^q(y)dm(y)=$$
$$=\int\limits_{f(D)}{\widetilde{\rho}}^{\,q}(y)dm(y)
\ge M_q\left(f(\Gamma)\right)\,.\,\Box$$

\end{proof}

\medskip
The {\it proof of  Theorem \ref{th1A}} directly follows from
Theorems \ref{th1} and \ref{th2} . $\Box$

\section{The analog of the V\"{a}is\"{a}l\"{a} inequality}
\setcounter{equation}{0}

The following result generalizes the well--known V\"{a}is\"{a}l\"{a}
inequality for the mappings with bounded distortion, see $\S\,9$ of
Ch. II in \cite{Ri} and Theorem 3.1 in \cite{Va$_2$}; see also
Theorem 4.1 in \cite{KO}.

\medskip
\begin{theorem}\label{th3.1}{\sl\,
Let $f:D\rightarrow {\Bbb R}^n$ be differentiable a.e. in $D,$ have
$N$-- and $N^{-1}$--pro\-per\-ties, and $L^{(2)}_q$--pro\-per\-ty.
Let $\Gamma$ be a curve family in $D,$ $\Gamma^{\,\prime}$ be a curve
family in ${\Bbb R}^n$ and  $m$ be a positive integer such that the
following is true. Suppose that for every curve $\beta:I\rightarrow
D$ in $\Gamma^{\,\prime}$ there are curves
$\alpha_1,\ldots,\alpha_m$ in $\Gamma$ such that $f\circ
\alpha_j\subset \beta$ for all $j=1,\ldots,m,$ and for every $x\in
D$ and all $t\in I$ the equality $\alpha_j(t)=x$ holds at most
$i(x,f)$ indices $j.$ Then
%
%\begin{equation}\label{equa8}
$$M_q(\Gamma^{\,\prime} )\quad\le\quad \frac{1}{m}\quad\int\limits_D
K_{I, q}\left(x,\,f\right)\cdot \rho^q (x)\quad dm(x)$$
%
%\end{equation}
%
for every $\rho \in {\rm \,adm}\,\Gamma.$}
\end{theorem}

{\it Proof.} Let $B_0$ and $C_{k}^{*}$  be as in Proposition
\ref{pr1}.  Setting by induction $B_1\,=\,C_{1}^{*},$
$B_2\,=\,C_{2}^{*}\setminus B_1\ldots\,,$
$$B_k\quad=\quad C_{k}^{*}\setminus \bigcup\limits_{l=1}^{k-1}B_l\,,$$
we obtain a  countable covering of $D$ consisting of mutually
disjoint Borel sets $B_k,$ $k=1,2,\ldots,$ with $m(B_0)=0,$
$B_0:=D\setminus\bigcup\limits_{k=1}^{\,\infty}B_k.$ By the
construction and $N$--pro\-per\-ty, $m\left(f(B_0)\right)=0.$

Thus $\widetilde{\gamma}^{\,0}(s)\not\in f(B_0)$ for a.e. $s$ and
$q$--a.e. closed curves $\widetilde{\gamma}$ in $f(D)$  by
Proposition \ref{pr2}; here $\widetilde{\gamma}^{\,0}(s)$ is a
normal representation of $\widetilde{\gamma}(s).$ Besides that a
curve $\gamma^{\,*}$, which  is the $f$--re\-p\-re\-sen\-ta\-tion of
$\gamma$,  is absolutely continuous for $q$--a.e.
$\widetilde{\gamma}=f\circ\gamma.$ Here the
$f$--re\-p\-re\-sen\-ta\-tion $\gamma^{\,*}$ of $\gamma$ is
well--defined for $q$--a.e. curves $\widetilde{\gamma}=f\circ\gamma$
(see Proposition \ref{pr3}).

Now let $\Gamma_1$ be a family of all (locally rectifiable) curves
$\gamma_1\in \Gamma^{\,\prime}$ for which there exists a closed
subcurve $\beta_1,$ $\beta_1=f\circ\alpha_1,$ such that the
$f$--re\-p\-re\-sen\-ta\-tion $\alpha_1^{\,*}$ of $\alpha_1$ either
is not rectifiable or is not absolutely continuous. Denote by
$\Gamma_2$ the family of all closed curves $\gamma_2,$
$\gamma_2=f\circ\alpha_2,$ such that the
$f$--re\-p\-re\-sen\-ta\-tion $\alpha_2^{\,*}$ of $\alpha_2$ either
is not rectifiable or is not absolutely continuous. We have proved
that $M_q(\Gamma_2)=0.$ On the other hand $\Gamma_1>\Gamma_2,$
consequently, $M_q(\Gamma_1)\le M_q(\Gamma_2)=0.$

Let $\rho\,\in\,{\rm  \,adm}\,\Gamma$ and
\begin{equation}\label{equa9}
\widetilde{\rho}(y)\quad=\quad\frac{1}{m}\cdot
\chi_{f\left(D\setminus B_0
\right)}(y)\sup\limits_{C}\sum\limits_{x\,\in\,C}\rho^*(x)\,,
\end{equation}
where
$$\rho^*(x)\,=\,\left\{\begin{array}{rr}
\rho(x)/l\left(f^{\prime}(x)\right), &   x\in D\setminus B_0,\\
0,  &  x\in B_0
\end{array}
\right.$$
and $C$ runs over all subsets of $f^{-1}(y)$ in $D\setminus B_0$
such that ${\rm card}\,C\le m.$ Note that
\begin{equation}\label{equa10}
\widetilde{\rho}(y)\quad=\quad\frac{1}{m}\cdot
\sup\sum\limits_{i=1}^s \rho_{k_i}(y)\,,
\end{equation}
where $\sup$ in (\ref{equa10}) is taken over all
$\left\{k_{i_1},\ldots,k_{i_s}\right\}$  such that $k_i\in\,{\Bbb
N},$ $k_i\ne k_j$ if $i\ne j,$ all $s\le m$  and
$$
\rho_k(y)\,=\left\{\begin{array}{rr}
\,\rho^{*}\left(f_k^{-1}(y)\right), &   y\in f(B_k),\\
0,  &  y\notin f(B_k)
\end{array}
\right. \,,$$
where $f_k=f|_{B_k},$ $k=1,2,\ldots\,$ is injective and $f(B_k)$ is
Borel. Thus, the function  $\widetilde{\rho}(y)$ is Borel, see e.g.
2.3.2 in \cite{Fe}. %and Theorem I (8.5) in \cite{Sa}.

Suppose that $\beta$ is a curve in $\Gamma^{\,\prime}.$ There exist
curves $\alpha_1,\ldots,\alpha_m$ in $\Gamma$ such that $f\circ
\alpha_j\subset \beta$ and for all $x\in D$ and $t$ the equality
$\alpha_j(t)=x$ holds for at most $i(x,f)$ indices $j.$ We show that
$\widetilde{\rho}\,\in\,{\rm }\,\,{\rm
adm}\,\Gamma^{\,\prime}\setminus \Gamma_0,$ where $M_q(\Gamma_0)=0.$
Without loss of generality we may consider that all of the curves
$\beta$ of $\Gamma^{\,\prime}$ are locally rectifiable, see Section
6 in \cite{Va$_1$}, p. 18. Now we suppose that $\beta$ is closed. We
may consider that $\beta^0(t)\not\in f(B_0)$ for a.e. $t\in [0,
l(\beta)],$ where $\beta^0:[0, l(\beta)]\rightarrow {\Bbb R}^n$ is a
normal representation of $\beta,$ i.e., $\beta(t)=\beta^0\circ
l_{\beta}(t).$ Denote by $\alpha_j^{\,*}(t):I_j\rightarrow D$ the
corresponding $f$--rep\-re\-sen\-ta\-tion of $\alpha_j$ with respect
to $\beta,$ i.e., $\alpha_j(t)=\alpha^*_j\circ l_{\beta}(t),$
$f\circ \alpha^*_j\subset \beta^0.$ Let
$$h_j(t)=\rho^{\,*}\left(\alpha^*_j(t)\right)\chi_{I_j}(t)\,,\quad t\in [0, l(\beta)]\,,\quad
J_t:=\{j:t\in I_j\}\,.$$ Since $\beta^0(t)\not\in f(B_0)$ for a.e.
$t\in [0, l(\beta)],$ the points $\alpha^{\,*}_j(t)\in
f^{\,-1}(\beta^0(t)),$ $j\in J_t,$ are distinct for a.e. $t.$ By the
definition of $\widetilde{\rho}$ in (\ref{equa9}),
\begin{equation}\label{eq6.1}
\widetilde{\rho}(\beta^0(t))\ge \frac{1}{m}\cdot\sum\limits_{j=1}^m
h_j(t)
\end{equation}
for a.e. $t\in [0, l(\beta)].$
By (\ref{eq6.1}) we have that
$$\int\limits_{\beta}\widetilde{\rho}(y)|dy|=\int\limits_{0}^{l(\beta)}
\widetilde{\rho}(\beta^0(t))dt\ge
$$
%\begin{equation}\label{eq6.2}
$$\ge \frac{1}{m}\cdot\sum\limits_{j=1}^m
\int\limits_{0}^{l(\beta)}h_j(t) dt=
\frac{1}{m}\cdot\sum\limits_{j=1}^m
\int\limits_{I_j}\rho^{\,*}\left(\alpha^*_j(t)\right)dm_1(t)\,.$$
%\end{equation}
%
Now we show that
\begin{equation}\label{eq6.4}
\int\limits_{I_j}\rho^{\,*}\left(\alpha^*_j(t)\right)dm_1(t)\ge 1
\end{equation}
for $q$--a.e. curve $\beta\in \Gamma^{\,\prime}.$ Since
$\beta^{\,0}(t)$ is rectifiable, $\beta^{\,0}(t)$ is differentiable
for a.e. $t\in I.$ Besides that, the curve $\alpha_j^{\,*}$ from the
$f$--re\-p\-re\-sen\-ta\-tion of $\beta$ is absolutely continuous
for $q$--a.e. $\beta$ by Proposition \ref{pr3}. Since
$\beta^0(t)\not\in f(B_0)$ for a.e. $t\in [0, l(\beta)],$ we have
$\alpha_j^{\,*}(t)\not\in B_0$ at a.e. $t\in I_j.$ Thus, the
derivatives $f^{\,\prime}\left(\alpha_j^{\,*}(t)\right)$ and
$\alpha_j^{\,*\prime}(t)$ exist for a.e. $t\in I_j.$ Taking into
account the formula of the derivative of the superposition of
functions, and that the modulus of the derivative of the curve by
the natural parameter equals 1, we have
$$1=\left|\left(f\circ \alpha_j^{\,*}\right)^{\,\prime}(t)\right|=
\left|f^{\,\prime}\left(\alpha_j^{\,*}(t)\right)\alpha_j^{\,*\prime}(t)\right|=$$
\begin{equation}\label{eq6.6}
=\left|f^{\,\prime}\left(\alpha_j^{\,*}(t)\right)\cdot\frac{\alpha_j^{\,*\prime}(t)}{|\alpha_j^{\,*\prime}(t)|}\right|\cdot
|\alpha_j^{\,*\prime}(t)|\ge
l\left(f^{\,\prime}\left(\alpha_j^{\,*}(t)\right)\right)\cdot
|\alpha_j^{\,*\prime}(t)|\,.\end{equation}
It follows from (\ref{eq6.6}) that
\begin{equation}\label{eq6.5}
\rho^{\,*}\left(\alpha^*_j(t)\right)=\frac{\rho(\alpha^*_j(t))}
{l\left(f^{\,\prime}\left(\alpha_j^{\,*}(t)\right)\right)}\ge
\rho(\alpha^*_j(t))\cdot |\alpha_j^{\,*\prime}(t)|\,.
\end{equation}
From (\ref{eq6.5}) by absolutely continuously of $\alpha_j^{\,*},$
definition of $\rho$ and Theorem 4.1 in \cite{Va$_1$} we have
\begin{equation}\label{eq6.7}
1\le
\int\limits_{\alpha_j}\rho(x)|dx|=\int\limits_{I_j}\rho\left(\alpha^*_j(t)\right)
\cdot |\alpha_j^{\,*\prime}(t)|\,dm_1(t)\le
\int\limits_{I_j}\rho^{\,*}\left(\alpha^*_j(t)\right)dm_1(t)\,.
\end{equation}
Now inequality (\ref{eq6.4}) directly follows from (\ref{eq6.7}).
Next we have that $\int\limits_{\beta}\widetilde{\rho}(y)|dy|\ge 1$
for $q$--a.e. closed curve $\beta$ of $\Gamma^{\,\prime}.$ The case
of  the arbitrary curve $\beta$ can be got from the taking of $\sup$
in $\int\limits_{\beta^{\,\prime}}\widetilde{\rho}(y)|dy|\ge 1$ over
all closed subcurves $\beta^{\,\prime}$ of $\beta.$
Thus, $\widetilde{\rho}\,\in\,{\rm }\,\,{\rm
adm}\,\Gamma^{\,\prime}\setminus \Gamma_0,$ where $M_q(\Gamma_0)=0.$
Hence
\begin{equation}\label{equa13}
M_q\left(\Gamma^{\,\prime\,}\right)\quad
\le\quad\int\limits_{f(D)}\widetilde{\rho}\,^q(y)\,\,dm(y)\,.
\end{equation}
By 3.2.5 for $m=n$ in \cite{Fe}, we obtain that
\begin{equation}\label{equa14}
\int\limits_{B_k}K_{I,
q}(x,\,f)\cdot\rho^q(x)\,\,dm(x)\quad=\quad\int\limits_{f(D)}
\rho^q_k(y)\,dm(y)\,.
\end{equation}
By H\"{o}lder inequality for series,
\begin{equation}\label{equa16}
\left(\frac{1}{m}\cdot\sum\limits_{i=1}^{s}\rho_{k_i}(y)\right)^q\quad\le\quad
\frac{1}{m}\cdot \sum\limits_{i=1}^{s}\,\rho^q_{k_i}(y)
\end{equation}
for each $1 \le s \le m$ and every $k_1,\ldots,k_s,$ $k_i\in {\Bbb
N},$ $i=1,2,\ldots,$ $k_i\ne k_j$ if $i\ne j.$

Finally, by Lebesgue positive convergence theorem, see Theorem
I.12.3 in \cite{Sa}, we conclude from (\ref{equa13})--(\ref{equa16})
that
%
%\begin{equation}\label{equa15}
$$\frac{1}{m}\cdot\int\limits_{D}K_{I,
q}(x,\,f)\cdot\rho^q(x)\,\,dm(x)\quad
=\quad\frac{1}{m}\cdot\int\limits_{f\left(D\right)}\,\sum\limits_{k=1}^{\infty}
\rho_k^q(y)\,dm(y)\quad\ge$$
%\end{equation}
%
$$\ge\quad\frac{1}{m}\cdot\int\limits_{f\left(D\right)}
\sup\limits_{\left\{k_1,\ldots,k_s\right\},\, k_i\in {\Bbb N},\atop
k_i\ne k_j \,{\rm if}\, i\ne j}\sum\limits_{i=1}^s
\rho^q_{k_i}(y)\,dm(y)\quad\ge\quad
\int\limits_{f\left(D\right)}\,\widetilde{\rho}^{\,q}(y)\,dm(y)\quad\ge\quad
M_q(\Gamma^{\,\prime\,})\,.$$
The proof is complete. $\Box$

\section{Applications}

\setcounter{equation}{0}

Let $f:D \rightarrow {\Bbb R}^n$ be a discrete open mapping, $\beta:
[a,\,b)\rightarrow {\Bbb R}^n$ be  a curve and
$x\in\,f^{-1}\left(\beta(a)\right).$ A curve $\alpha:
[a,\,c)\rightarrow D$ is called a {\it maximal $f$--lifting} of
$\beta$ starting at $x,$ if $(1)\quad \alpha(a)=x\,;$ $(2)\quad
f\circ\alpha=\beta|_{[a,\,c)};$ $(3)$\quad if $c<c^{\prime}\le b,$
then there is no curve $\alpha^{\prime}: [a,\,c^{\prime})\rightarrow
D$ such that $\alpha=\alpha^{\prime}|_{[a,\,c)}$ and $f\circ
\alpha^{\,\prime}=\beta|_{[a,\,c^{\prime})}.$  By assumption on $f$ one implies that every  curve $\beta$ with $x\in
f^{\,-1}\left(\beta(a)\right)$ has  a maximal $f$--lif\-ting
starting at the point $x,$ see Corollary II.3.3 in \cite{Ri}.

\medskip
Let $x_1,\ldots,x_k$ be $k$ different points of
$f^{-1}\left(\beta(a)\right)$ and let
$$m=\sum\limits_{i=1}^k i(x_i,\,f)\,.$$
We say that the sequence $\alpha_1,\dots,\alpha_m$ is a {\it maximal
sequence of $f$--liftings of $\beta$ starting at points
$x_1,\ldots,x_k,$} if

$(a)$\quad each $\alpha_j$ is a maximal $f$--lifting of $\beta,$

$(b)\quad {\rm card}\,\left\{j:a_j(a)=x_i\right\}= i(x_i,\,f),\quad
1\le i\le k\,,$

$(c)\quad {\rm card}\,\left\{j:a_j(t)=x\right\}\le i(x,\,f)$ for all
$x\in D$ and for all $t.$

Let $f$ be a discrete open mapping and $x_1,\ldots,x_k$ be distinct
points in $f^{\,-1}\left(\beta(a)\right).$ Then $\beta$ has a
maximal sequence of $f$--lif\-tings starting at the points
$x_1,\ldots,x_k,$ see Theorem II.3.2 in \cite{Ri}.

\medskip
A domain $G\subset D,$ $\overline{G}\subset D,$ is said to be a {\it
normal domain of $f,$} if $\partial f(G)=f\left(\partial G\right).$
If $G$ is a normal domain, then $\mu(y, f, G)$ is a constant for
$y\in f(G).$ This constant will be denoted by $\mu(f, G).$ Let
$f:D\rightarrow {\Bbb R}^n$ be a discrete open mapping, then $\mu(f,
G)=N(f,G)$ for every normal domain $G\subset D,$  see e.g.
Proposition I.4.10 in \cite{Ri}. We need bellow the following
statement, see Corollary II.3.4 in \cite{Ri}.

\medskip
\begin{lemma}\label{lem2.3}
{\sl\, Let $f:G\rightarrow {\Bbb R}^n$ be a discrete open mapping,
$G$  be a normal domain, $m=N(f,G),$ $\beta:[a,b)\rightarrow f(G)$ be a
curve. Then there exist curves $\alpha_j:[a,b)\rightarrow G,$ $1\le
j\le m,$ such that:

1)\quad $f\circ \alpha_j\,=\,\beta,$

2)\quad ${\rm card}\,\left\{j:\alpha_j(t)=x\right\}= i(x,\,f)$ for
$x\in\,G\cap f^{-1}(\beta(t)),$

3)\quad $|\alpha_1|\cup\ldots\cup|\alpha_m|=G\cap f^{-1}(|\beta|).$}
\end{lemma}

\medskip
We also adopt the following conventions. Given $q\ge 1,$ a family of
curves $\Gamma$ in ${\Bbb R}^n,$ denote
$$M_{q, K_{I, q}(\cdot,\,f)}(\Gamma)\,=\, \inf\limits_{\rho\,\in\,{\rm adm}\,\Gamma }
\int\limits _{{\Bbb R}^n}\,\rho^q(x)\,K_{I, q}(x,f)\,dm(x)\,.$$

\medskip
The following result holds.

\medskip
\begin{corollary}{}\label{th4.1}{\sl\,
Let $f:D\rightarrow {\Bbb R}^n$ be an open discrete mapping, which
is differentiable a.e. in $D,$ have  $N$-- and
$N^{-1}$--pro\-per\-ties, and $L^{(2)}_q$--pro\-per\-ty. Suppose
that $G$ is a normal domain for $f,$ $\Gamma^{\,\prime}$ is  a curve
family in $G^{\,\prime}=f(G)$ and  $\Gamma$ is  a family consisting of
all curves $\alpha$ in $G$ such that $f\circ \alpha\subset
\Gamma^{\,\prime}$ and $m=N(f,G).$
Then
$$M_q(\Gamma^{\,\prime})\quad\le\quad\frac{1}{N(f,G)}\quad\int\limits_G
K_{I, q}\left(x,\,f\right)\cdot \rho^q (x)\ \ dm(x)$$
for every $\rho \in {\rm \,adm}\,\Gamma.$ Moreover,
$$M_q(\Gamma^{\,\prime})\quad\le\quad\frac{1}{N(f,G)}\quad
M_{q, K_{I, q}(\cdot,\,f)}(\Gamma)\,.$$}
\end{corollary}

The proof directly follows from Theorem \ref{th3.1} and  Lemma
\ref{lem2.3}. $\Box$

\medskip
Following section II.10 in \cite{Ri}, a {\it condenser} is a pair
$E=(A, C)$ where $A\subset {\Bbb R}^n$ is open and $C$ is non--empty
compact set contained in $A.$ A condenser $E=(A, C)$ is said to be
in a domain $G$ if $A\subset G.$ For a given condenser
$E=\left(A,\,C\right)$ and $q\ge 1,$ we set
%
%\begin{equation}\label{equ5}
$${\rm cap}_q\,E\quad=\quad{\rm
cap}_q\,\left(A,\,C\right)\quad=\quad\inf \limits_{u\in
W_0\left(E\right) }\quad\int\limits_{A}\,|\nabla u|^q dm(x)$$
%\end{equation}
%
%
where $W_0(E)=W_0(A,\,C)$ is the family of all non--negative
functions $u:A\rightarrow {\Bbb R}^1$ such that (1)\quad $u\in
C_0(A),$\quad(2)\quad $u(x)\ge 1$ for $x\in C,$ and (3)\quad $u$ is
$ACL.$ In the above formula
$|\nabla u|={\left(\sum\limits_{i=1}^n\,{\left(\partial_i
u\right)}^2 \right)}^{1/2},$ and ${\rm cap}_q\,E$ is called {\it
$q$--ca\-pa\-ci\-ty} of the condenser $E,$ see Section II.10 in
\cite{Ri}.

\medskip
Let $E=(A,C)$ be a condenser and $\omega$ be  a nonnegative
measurable function.  We define the {\it $\omega$--weighted
capacity} of $E$ by setting
\begin{equation}\label{equa19}
{\rm cap}_{q, \omega}\,E\,=\,{\rm cap}_{q, \omega}\,(A,C)\,=\,\inf\,
\int\limits_{A}\,|\nabla u(x)|^q\,\omega(x)\,dm(x)\,,
\end{equation}
where $\inf$ in (\ref{equa19}) is taken over all functions $u\in
C_0^{\infty}(A)$ and $u\ge 1$ on $C.$

Given a mapping $f:D\rightarrow {\Bbb R}^n$ and a condenser
$E=(A,C),$ we call
$$M(f,C)\quad=\quad \inf\limits_{y\,\in\, f(C)} \sum\limits_{x\,\in\, f^{\,-1}(y)\cap C}
\,i(x,f)$$
the {\it minimal multiplicity} of $f$ on $C.$

\medskip
We need the following statement, see Proposition II.10.2 in
\cite{Ri}.
\begin{lemma}\label{lem2.2}
{\sl\, Let $E=(A,\,C)$ be a condenser in ${\Bbb R}^n$ and let
$\Gamma_E$ be the family of all curves of the form
$\gamma:[a,\,b)\rightarrow A$ with $\gamma(a)\in C$ and
$|\gamma|\cap\left(A\setminus F\right)\ne\varnothing$ for every
compact $F\subset A.$ Then ${\rm
cap}_q\,E=M_q\left(\Gamma_E\right).$}
\end{lemma}

\medskip
\begin{theorem}\label{th4.2}
{\sl\, Let $f:D\rightarrow {\Bbb R}^n$ be an open discrete mapping,
which is differentiable a.e. in $D,$ have  $N$-- and
$N^{-1}$--pro\-per\-ties, and $L^{(2)}_q$--pro\-per\-ty. Suppose that
$E=(A,C)$ is  a condenser in $D.$ Then
\begin{equation}\label{equa20}
{\rm cap}_q\,f(E)\quad \le \quad \frac{1}{M(f,C)}\,\,{\rm cap}_{q,
K_{I, q}(\cdot,\,f)}\,E\,.
\end{equation}}
\end{theorem}

\medskip
\begin{proof} Since   $E=(A,C)$  is  a condenser in $D,$ then $f(E)=\left(f(A), f(C)\right)$
is a condenser in $f(D).$ Let $\Gamma_E$ and $\Gamma_{f(E)}$ be
curve families such as in Lemma \ref{lem2.2}. Set $m=M(f,C).$ Let
$\beta:[a, b)\rightarrow f(A)$ be a curve in $\Gamma_{f(E)}.$ Then
$C\cap f^{\,-1}\left(\beta(a)\right)$ contains points
$x_1,\ldots,x_k$ such that
$$m^{\,\prime}\quad=\quad\sum\limits_{l=1}^k i(x_l, f)\quad\ge\quad m\,.$$
By Theorem II.3.2 in \cite{Ri}, there is a maximal sequence of
$f|A$--liftings $\alpha_j:[a,c_j)\rightarrow D$ of $\beta,$ $1\le
j\le m^{\,\prime},$ starting at the points $x_1,\ldots, x_k.$ Then
each $\alpha_j$ belongs to $\Gamma_E.$ It follows that
$\Gamma=\Gamma_E$ and $\Gamma^{\,\prime}=\Gamma_{f(E)}$ satisfy the
Theorem \ref{th3.1}. Hence by Lemma \ref{lem2.2}
\begin{equation}\label{equa21}
{\rm cap}_q\,f(E)\quad \le \quad \frac{1}{M(f,C)}\qquad M_{q, K_{I,
q}(\cdot,\,f)}(\Gamma_E)\,.
\end{equation}
Finally, (\ref{equa20}) follows from (\ref{equa21}) because
%
%\begin{equation}\label{equa22}
$$M_{q, K_{I, q}(\cdot,\,f)}(\Gamma_E)\quad\le \quad {\rm cap}_{q,
K_{I, q}(\cdot,\,f)}\,E\,,$$
%\end{equation}
as it is easily seen by considering $\rho(x)=|\nabla u(x)|$ for a
given test function $u$ in ${\rm cap}_{q, K_{I, q}(\cdot,\,f)}\,E.$
In fact, let $\gamma\in \Gamma_E$ be a locally rectifiable curve,
$s$ be a natural parameter on $\gamma,$ $\gamma^{\,0}$ be a
normal representation of $\gamma$ and let $s_0\in (0, l(\gamma)),$
where $l(\gamma)$ denotes the length of the curve $\gamma.$ Using a
geometrical sense of the gradient, for the function $\rho(x)=|\nabla
u(x)|,$ we have
$$\int\limits_\gamma\rho(x)|dx|=\int\limits_\gamma|\nabla u(x)||dx|=
\int\limits_{0}^{l(\gamma)}|\nabla u(\gamma^{\,0}(s))|ds\ge
\int\limits_{0}^{l(\gamma)}\left|\frac{du(\gamma^{\,0}(s))}{ds}\right|ds\ge
$$
$$\ge \left|\int\limits_{0}^{s_0}\frac{du(\gamma^{\,0}(s))}{ds}\,ds\right|=\left|u(\gamma^{\,0}(s_0))
-u(\gamma^{\,0}(0))\right|\rightarrow
\left|u(\gamma^{\,0}(l(\gamma)))-u(\gamma^{\,0}(0))\right|=1$$
as $s_0\rightarrow l(\gamma).$ Here we have used that
$r(s)=u(\gamma^{\,0}(s))$ is absolutely continuous by parameter $s$
for every rectifiable closed curve $\gamma$ in $D$ because a
function $u\in C_0^{\,\infty}(A)$ is locally Lipschitzian. Thus,
$\rho(x)=|\nabla u(x)|\in {\rm adm}\,\Gamma_E,$ and we have
the desired conclusion (\ref{equa20}).
\end{proof}$\Box$

%=================Список литературы====================
%\end{fulltext}

{\bf \noindent Ruslan Salimov and Evgeny Sevost'yanov} \\
Institute of Applied Mathematics and Mechanics,\\
National Academy of Sciences of Ukraine, \\
74 Roze Luxemburg str., 83114 Donetsk, UKRAINE \\
Phone: +38 -- (062) -- 3110145, \\
Email: ruslan623@yandex.ru, esevostyanov2009@mail.ru
\end{document}